\documentclass{amsart}

\usepackage[english]{babel}     
\usepackage{amsmath}      
\usepackage{amssymb}     
\usepackage{mathrsfs} 
\usepackage{stmaryrd}
\usepackage{hyperref}
\usepackage[utf8]{inputenc}

\newtheorem{theorem}{Theorem}[section]
\newtheorem{lemma}[theorem]{Lemma}

\theoremstyle{definition}
\newtheorem{definition}[theorem]{Definition}
\newtheorem{example}[theorem]{Example}

\newtheorem{proposition}[theorem]{Proposition}

\theoremstyle{remark}
\newtheorem{remark}[theorem]{Remark}

\numberwithin{equation}{section}

\newcommand{\att}{\big \vert}
\newcommand{\As}		{{\operatorname{\mathscr{A}}}}
\newcommand{\pr}			{{\operatorname{\mathsf{pr}}}}
\newcommand{\av}		{{\operatorname{\mathrm{av}}}}
\newcommand{\Vol}			{{\operatorname{\mathsf{Vol}}}}
\newcommand{\N}			{{\operatorname{\mathsf{N}}}}

\newcommand{\diff}		{{\operatorname{\mathrm{diff}}}}
\newcommand{\cont}		{{\operatorname{\mathrm{cont}}}}
\newcommand{\loc}		{{\operatorname{\mathrm{loc}}}}
\newcommand{\ncdiff}	{{\operatorname{\mathrm{n.c.diff}}}}
\newcommand{\acts}            {\mathbin{\triangleright}}
\newcommand{\form}	{{\operatorname{\llbracket \lambda \rrbracket}}}

\DeclareMathOperator{\im}{Im}
\DeclareMathOperator{\supp}{supp}
\DeclareMathOperator{\inv}{inv}
\DeclareMathOperator{\id}{id}
\DeclareMathOperator{\Hom}{Hom}
\DeclareMathOperator{\Der}{Der}
\DeclareMathOperator{\InnDer}{InnDer}

\DeclareMathOperator{\ih}{i}
\DeclareMathOperator{\HC}{HC}
\DeclareMathOperator{\HH}{HH}

\begin{document}

\title{Invariant Hochschild cohomology of smooth functions}

\author{Lukas Miaskiwskyi}
\curraddr{EWI/DIAM, P.O.Box 5031, 2600 GA Delft, The Netherlands
}
\email{l.t.miaskiwskyi@tudelft.nl}


\date{2-12-2020}


\begin{abstract}
Given an action of a Lie group on a smooth manifold, we discuss the induced action on the Hochschild cohomology of smooth functions, and notions of invariance on this space. Depending on whether one considers invariance of cochains or invariance of cohomology classes, two different spaces of invariants arise. We perform a general comparison of these notions, give an interpretation of the lower orders of the invariant cohomology spaces and conclude as our main result that for proper group actions both spaces are isomorphic. As a corollary and a geometric interpretation, an invariant version of the Hochschild-Kostant-Rosenberg theorem is given, identifying the cohomology of invariant cochains with invariant multivector fields. 
Using this theorem, we shortly discuss the invariant Hochschild cohomology in the case of homogeneous spaces. 
\end{abstract}

\maketitle

\section{Introduction}
Hochschild cohomology, initially investigated by Hochschild in \cite{hochschild1945cohomology} and significantly developed by Gerstenhaber in \cite{gerstenhaber1,gerstenhaber2,gerstenhaber3,gerstenhaber4,gerstenhaber5}, is found to be a valuable tool in deformation theory \cite{kontsevich1,kontsevich2}. Precisely, if one is interested in formal deformations of a certain algebra $\As$, it is well-known that the lower orders of the corresponding Hochschild cohomology $\HH^\bullet(\As, \As)$ characterize both the equivalence classes of infinitesimal deformations and the obstructions to order-by-order continuation of an arbitrary formal deformation, see \cite{gerstenhaber2}.

As one motivating example, in formal de\-for\-ma\-tion quan\-ti\-za\-tion, o\-ri\-gi\-nal\-ly con\-si\-de\-red in \cite{bayen1977quantum}, one is interested in constructing deformations of the algebra of smooth functions $C^\infty(M)$ on some manifold $M$ in order to equip this commutative algebra with a new, non-commutative multiplication. Indeed, here one is able to explicitly calculate continuous Hochschild cohomology by means of the Hochschild-Kostant-Rosenberg theorem, whereby the cohomology groups can be identified with multivector fields on the manifold (see \cite{HKR} for the original formulation of this theorem in a purely algebraic setting and \cite{kontsevich2} for the differential case). However, to gain a more refined, maybe even finite-dimensional classification, one needs to add further information to the setting. \footnote{One such piece of information is the choice of a Poisson structure on the manifold, which one requires deformations to respect in some sense. Notable mentions here are Fedosov's results \cite{fedosov1994simple} in the symplectic, and Kontsevich's results \cite{kontsevich1}, \cite{kontsevich2} in the general Poisson case.}

In the following, we want to consider the case where desirable cochains are invariant under a certain symmetry of the manifold, here modeled as the action of a Lie group. This idea is again motivated by deformation quantization, where the manifold is considered a physical phase space, which in practice often admits certain symmetries (e.g. translational, rotational, scaling). The requirement of a given deformation to also reflect this symmetry then naturally leads one to consider invariant Hochschild cochains.

For this document, which began as the author's Master thesis, we begin by describing the Hochschild cohomology in the most general setting in Section \ref{SectionHochschildCohomo}, while noting the most important specialties that arise in the case of the smooth functions. In Section \ref{SectionActionsAndInvariant} we describe how to equip this space with reasonable notions of invariance, and end up with two different notions: invariance of \emph{Hochschild cochains} on one hand and invariance of \emph{Hochschild cohomology classes} on the other. The invariant cochains naturally inherit the structure of a cochain complex, and we denote the \emph{cohomology of invariant cochains} by $\HH^\bullet_G(\As,\N)$, whereas the \emph{space of invariant classes} of the original cohomology is denoted $\HH^\bullet(\As,\N)^G$. There is a natural morphism
\begin{align}
\iota : \HH^\bullet_G(\As,\N) \to \HH^\bullet(\As,\N)^G
\end{align}
relating both notions. Injectivity and surjectivity of this map yield nontrivial statements about how the differential interacts with the group action, and will be our focus of investigation in this document.  In Section \ref{SectionLowOrders} we give explicit interpretations of the lower orders of the cohomology of invariant cochains, analogous to the ones given for the usual Hochschild cohomology in the context of deformation theory.

In Section \ref{SectionProperHKR} we prove our main result, which is that $\iota$ is an isomorphism in the case where $\As = C^\infty(M)$ is the space of smooth functions on a manifold and the group is acting properly on $M$. Properness of the action ensures the existence of invariant partitions of unity and therefore an averaging operator on the space of cochains, which will give us injectivity of the map $\iota$. The surjectivity on the other hand is a straightforward consequence of the classical Hochschild-Kostant-Rosenberg theorem and does not rely on any constraints on the group action. 

In particular, this isomorphism allows us to formulate an invariant Hochschild-Kostant-Rosenberg theorem, relating cohomology of invariant cochains to invariant multivector fields on the given manifold.
A few technical calculations  are moved to the appendix.

Note that the interplay of Hochschild (co-)homology and group actions has been studied extensively by others, from various different angles: One may for example name \cite{brodzki2017periodic}, \cite{brylinski1987cyclic}, and \cite{block1994equivariant}, which treat for given discrete or compact groups $G$ the problem of calculating equivariant cyclic/Hochschild homology. In particular, in \cite{brylinski1987cyclic} and \cite{block1994equivariant} the homology of the crossed product algebra $C^\infty(M \times G)$ equipped with a convolution product is discussed, which, as proven in \cite{brylinski1987algebras,brylinski1987cyclic}, can be identified with the homology of invariant Hochschild cochains for compact $G$, see also \cite[Thm 2.1]{block1994equivariant}. Also, \cite{neumaier2006homology} investigates invariant Hochschild cohomology of the crossed-product algebra and reaches a result similar to ours in the restriction of proper \'etale groupoids. In the context of the current paper, this may motivate further examination of the Hochschild cohomology of such crossed-product algebras and its relation to the  notions of invariant cohomology discussed in the following.  

The aforementioned literature does not appear to cover the precise results given here. In particular, we remark that our results are of a different flavour than the results of Hochschild homology of the crossed product algebras, the latter of which, for example, Brylinski considers in \cite{brylinski1987cyclic}. There, for instance, it is found that the Hochschild homology of a convolution algebra of compactly supported smooth functions $C^\infty_c(M \times G)$ reduces to the Hochschild homology of functions on the quotient space when $G$ is compact and acts freely. 
On the other hand, our notion of invariant Hochschild cohomology can be identified with invariant multivector fields, which, in general, do not simply arise from multivector fields on the quotient space.

\section{Hochschild cohomology}
\label{SectionHochschildCohomo}
We begin by recalling some basic definitions about the Hochschild complex and its cohomology.
Throughout the next sections, we fix some field $\mathbb{K}$, an associative $\mathbb{K}$-algebra $\As$ and an $\As$-bimodule $\N$. While we want to keep the definitions general for now, our main results will concern the smooth functions $C^\infty(M) = \As = \N$ on a smooth manifold $M$ and $\mathbb{K} = \mathbb{R}$ or $\mathbb{C}$.

\begin{definition}[Hochschild complex] Define the \emph{space of Hoch\-schild co\-chains} via
\begin{equation}
\begin{aligned}
\HC^\bullet(\mathcal A, \N) &:= \bigoplus_{n=0}^\infty \HC^n(\mathcal A, \N) \\ &:= N \oplus \Hom_\mathbb{K}(\As,\N) \oplus \Hom_\mathbb{K}(\As^{\otimes 2},\N) \oplus \Hom_\mathbb{K}(\As^{\otimes 3},\N) \oplus\dots.
\end{aligned}
\end{equation}
Define the \emph{Hochschild (co-)differential} $\delta : \HC^\bullet(\As,\N) \to \HC^{\bullet + 1}(\As,\N)$ for elements $\phi \in \HC^n(\As,\N)$ via
\begin{equation}
\begin{aligned}
(\delta \phi)(a_0, \dots , a_n) 
:=& \,
a_0 \cdot \phi(a_1 , \dots , a_n) 
+ 
(-1)^{n+1} \phi(a_0 , \dots , a_{n-1}) \cdot a_n \\
&+ 
\sum_{i=0}^{n-1} (-1)^{i+1} \phi(a_0 , \dots , a_{i-1}, a_i a_{i+1} , a_{i+2} , \dots , a_n).
\end{aligned}
\end{equation}
We call the pair $(\HC^\bullet(\As,\N),\delta)$ the \emph{Hochschild complex} of $\As$ with values in $\N$, which is pictured via
\begin{align}
0 \to \N = \HC^0(\As,\N) \stackrel{\delta}{\to}
\HC^1(\As,\N)  \stackrel{\delta}{\to}
\HC^2(\As,\N) \to \dots
\end{align}
Remember that $\delta^2 = 0$. The \emph{Hochschild cohomology} $\HH^\bullet(\As,\N)$ is defined to be the usual graded vector space
\begin{align}
\HH^\bullet(\As,\N) := \frac{\ker \delta}{\im \delta}.
\end{align}
\end{definition}

\begin{remark}
In the case $\N = \As$, where $\As$ is regarded an $\As$-bimodule by algebra multiplication, this space can be given the structure of a super Lie algebra, and the cohomology can be equipped with the structure of a Gerstenhaber algebra. While we do not use this fact here, we want to mention that all further results about group actions and invariants will be compatible with the Lie algebra/Gerstenhaber structures, so any time we talk about morphisms, they can be regarded as morphisms in the category of either super Lie algebras or Gerstenhaber algebras.\\
If one is unfamiliar with these structures, the reader may simply consider everything to take place in the category of $\mathbb{K}$-vector spaces.
\end{remark}

In the case $\As = \N = C^\infty(M)$ of smooth functions on a smooth manifold $M$, which is the one we are most interested in, we can apply analytical methods to the Hochschild cochains if we restrict to ``analytically interesting'' cochains.  Hence, denote by $\HC^\bullet_{\cont}(C^\infty(M))$ the Hochschild complex of continuous cochains with respect to the Fr\'{e}chet to\-po\-lo\-gy on $C^\infty(M)$. Accordingly, denote by $\HH^\bullet_\cont(C^\infty(M))$ the corresponding cohomology.

\begin{remark}
Similarly, one can restrict to \emph{local}, \emph{differential}, and \emph{differential, vanishing on constants} cochains. From \cite[p.413f.]{waldmann} we cite
\begin{align*}
\HC^\bullet_{\ncdiff} (C^\infty(M)) \subset
\HC^\bullet_{\diff} (C^\infty(M)) \subset
\HC^\bullet_{\loc} (C^\infty(M)) \subset
\HC^\bullet_{\cont} (C^\infty(M)),
\end{align*}
and they all have well-defined cohomologies, which have analogous notation.
\end{remark}

\section{Group actions and invariant cohomology}
\label{SectionActionsAndInvariant}
\subsection{Definitions}
We will now define our notions of group actions on the different spaces and according invariant spaces.
Let $G$ be some group acting on both $\As$ and $\N$, in the sense that
\begin{itemize}
\item[1)] the group acts on $\As$ via algebra isomorphisms,
\item[2)] the group acts on $\N$ via isomorphisms with respect to the Abelian group structure of $\N$, and, denoting the actions by $\acts$,
\begin{align}
\label{EquationModuleAction}
g \acts(a \cdot n) = (g \acts a) \cdot (g \acts n), \quad
g \acts(n \cdot a) = (g \acts n) \cdot (g \acts a),
\end{align}
for all $g \in G$, $a \in \As$, $n \in \N$.
\end{itemize}

Specifically in the previously mentioned case $\As = \N = C^\infty(M)$, if we assume the existence of an action on the manifold, desired actions on $\As$ and $\N$ are induced by pullback.

We can lift these actions to the space of Hochschild cochains:

\begin{definition}[Space of invariant Hochschild cochains]
Given actions on $\As$ and $\N$ as above, define for all $n \in \mathbb{N}_0$
\begin{equation}
\begin{aligned}
\acts &: G \times \HC^n(\As,\N) \to \HC^n(\As,\N) \\
(g \acts \phi) (a_1, &\dots, a_n) := g \acts (\phi(g^{-1} \acts a_1, \dots, g^{-1} \acts a_n)).
\end{aligned}
\end{equation}

We define the \emph{space of invariant Hochschild cochains} via
\begin{align}
\HC_G^\bullet(\As,\N) :&= \left( \HC^\bullet(\As,\N) \right)^G =
\{\phi \in \HC^\bullet \mid g \acts \phi = \phi  \quad \forall g \in G\}.
\end{align}
\end{definition} 
The Hochschild differential commutes with the group action and as such, we can restrict the Hochschild differential to a map
\begin{align}
\delta : \HC_G^\bullet(\As,\N) \to \HC^\bullet_G(\As,\N).
\end{align}
Hence, the space of invariant cochains inherits the structure of a complex, and we can define the associated cohomology:

\begin{definition}[Invariant Hochschild complex]
The tuple $(\HC_G^\bullet (\mathcal{A},\N),\delta)$ is called the \emph{invariant Hochschild complex} of $\mathcal{A}$ and $\N$, equivalently pictured as the sequence
\begin{align}
0 \to \N^G = \HC_G^0(\As,\N) \stackrel{\delta}{\to}
\HC_G^1(\As,\N)  \stackrel{\delta}{\to}
\HC_G^2(\As,\N) \to \dots
\end{align}
One then declares the \emph{cohomology of invariant cochains} $\HH_G^\bullet(\As,\N)$ to be the graded vector space defined by
\begin{align}
\HH_G^\bullet(\As,\N) := \frac{\ker \delta \att _{\HC_G^\bullet(\As,\N)}}{\im \delta\att _{\HC_G^\bullet(\As,\N)}}.
\end{align}
\end{definition}

Like with standard Hochschild cohomology, for the case $\As = \N = C^\infty(M)$ denote the invariant complexes with analytical properties (continuous, local etc.) and the respective cohomologies with the according subscript, e.g. 
\begin{align}
\HC_{G,\cont}^\bullet(C^\infty(M)), \, \,
\HH_{G,\loc}^\bullet(C^\infty(M))
\end{align} and so on, whenever the action is compatible with the respective property.

One can also consider invariance on the level of \emph{cohomology} rather than on the cochains: The action on the cochains canonically descends to one on the equivalence classes, so one may define the space of invariant \emph{classes}
\begin{align}
\HH^\bullet(\As,\N)^G := \{ [\phi] \in \HH^\bullet(\As,\N) \mid g \acts [\phi] = [\phi] \quad \forall g \in G\}.
\end{align}
Analogous spaces can be defined in the case $\N = \As = C^\infty(M)$ for the analytical subcomplexes.

Note the difference between $\HH^\bullet_G(\As,\N)$ and $\HH^\bullet(\As,\N)^G$: when one wants to limit one's framework to invariant cochains, e.g. when considering invariant deformations in deformation theory, the interesting space to consider is the cohomology of invariant cochains $\HH^\bullet_G(\As,\N)$, as we will see in Section \ref{SectionLowOrders}.

 In general, however, we will see that this space cannot easily be related to the original Hochschild cohomology, hence, one is opening a whole new can of cohomological worms. The space of invariant classes $\HH^\bullet(\As,\N)^G$, in contrast, is simply a subspace of the original cohomology.

The natural question arises whether the two notions of invariance can be related. As coboundaries in $\HC^\bullet_G(\As,\N)$ can also be considered coboundaries in $\HC^\bullet(\As,\N)$, there is a well-defined map 
\begin{align}
\iota : \HH_G^\bullet(\As,\N) \to \HH^\bullet(\As,\N)^G, [\phi] \mapsto [\phi].
\end{align}
The main topic of this document will be proving that this map is an isomorphism in the case where we assume $\As = \N = C^\infty(M)$, proper actions, and a restriction to continuous cochains.

\section{Interpretation of the lower order cohomology spaces}
\label{SectionLowOrders}
As with non-invariant Hochschild cohomology, it is possible to give lower orders of invariant Hochschild cohomology an interpretation. One quickly finds 
\begin{align*}
\HH^0_G(\As,\N) &= Z(\N)^G , \\
\HH^1_G(\As,\N) &= \Der(\As,\N)^G/ \InnDer(\As,\N)^G
\end{align*}
in the lowest orders, where $Z$, $\Der$ and $\InnDer$ respectively denote the center, derivations and inner derivations of $N$ with respect to the $(\As,\As)$-bimodule structure, and the superscript $G$ denotes restricting to invariant objects. In the context of deformation theory, there are also direct interpretations for the orders 2 and 3, which we will translate into the invariant context in the following.

We quickly recall a few necessary definitions of deformation theory, which can be looked up in more detail in \cite[Chapter 6]{waldmann}. Denote the multiplication on $\As$ by $\mu_0 : \As \otimes \As \to \As$, and write $\As \form$ for the algebra of formal power series in $\As$ in a parameter $\lambda$. In the following we set $\HC^\bullet(\As) := \HC^\bullet(\As,\As)$.

\begin{definition}
A \emph{formal deformation} of the algebra $\As$ is an associative map $\mu : \As \form \otimes \As \form \to \As \form$ which can be written in the form
\begin{align}
\mu = \mu_0 + \lambda \mu_1 + \lambda^2 \mu_2 + \dots
\end{align}
with maps $\mu_i \in \HC^2(\As)$. 
A \emph{formal deformation up to order $k \in \mathbb{N}$} is defined analogously replacing $\As \form$ by $\As \form / \langle \lambda^{k+1} \rangle$.

Two formal deformations $\mu, \tilde \mu$ of the same algebra are \emph{equivalent} if there exists an algebra isomorphism $\Phi : (\As \form, \mu) \to (\As \form, \tilde \mu)$ which is the identity in the zeroth order. Such a $\Phi$ is called an \emph{equivalence transformation}
\end{definition}

Recall also the existence of the \emph{Gerstenhaber bracket} (see \cite{gerstenhaber1} for the original work on this), 
\begin{align}
[\cdot,\cdot] : \HC^\bullet(\As) \otimes \HC ^\bullet(\As) \to \HC^\bullet(\As),
\end{align} 
making the Hochschild complex into a Lie superalgebra (if one defines the grading to give $\phi \in \HC^i(\As,\As)$ the degree $i-1$). 


We do not want to give the lengthy, explicit formula for this bracket here, but let us remark that the Gerstenhaber bracket is a linear combination of partial compositions of its arguments. From this, one can derive that the bracket is $G$-equivariant, i.e. for all $\phi, \psi \in \HC^\bullet(\As)$:
\begin{align}
g \acts [\phi, \psi] = [g \acts \psi, g \acts \phi].
\end{align}

Now, let us describe the obstructions given by the spaces $\HH^2_G(\As)$ and $\HH^3_G(\As)$. For this, we consider the well-known statements about the non-invariant Hochschild cohomology, see for example \cite[p.402ff.]{waldmann}, and formulate them for the invariant setting. The proofs are immediate adaptations of the proofs for the non-invariant statements.

\begin{proposition}[Obstructions in $\HH_G^3(\As)$]
\label{PropositionObstruction1}
Denote by 
\begin{align}
\mu^{(k)} = \mu_0 + \dots + \lambda^k \mu_k
\end{align} 
an invariant formal deformation of $\mu_0$ up to order $k$, so $\mu_i \in \HC_G^2(\As,\As)$ for all $i = 1, \dots, k$. Then
\begin{align}
R_{k+1} = -\frac{1}{2} \sum_{l=1}^k [\mu_l , \mu_{k+1-l}]
\end{align}
is an invariant Hochschild 3-cocycle, and  $\mu^{(k)}$ can be continued to an invariant associative deformation of order $k+1$ if and only if $R_{k+1} = \delta \mu_{k+1}$ for some $\mu_{k+1} \in \HC_G^2(\As,\As)$. In this case $\mu^{(k+1)} := \mu^{(k)}  + \lambda^{k+1} \mu_{k+1}$ yields such a continuation.
\end{proposition}

\begin{proposition}[Obstructions in $\HH^2_G(\As)$] 
\label{PropositionObstruction2} Let $\mathbb{K}$ be of characteristic zero. Given two invariant formal deformations $\mu, \tilde \mu$ which are identical up to order $k$, their difference $\mu_{k+1} - \tilde \mu_{k+1}$ is an invariant cocycle. \\
Furthermore, there exists an equivalence transformation up to order $k+1$ of the form $S = \exp( \lambda^{k+1} [T_{k+1}, \cdot])$ with an invariant $T_{k+1} \in \HC_G^1(\As,\As)$ if and only if $\mu_{k+1} - \tilde \mu_{k+1} = \delta \tilde{T}_{k+1}$ for some $\tilde T_{k+1} \in \HC_G^1(\As,\As)$. In this case both operators $\tilde T_{k+1}, T_{k+1}$ can be chosen to be identical.
\end{proposition}

\begin{remark}
Note that the formulation of \ref{PropositionObstruction2} in \cite{waldmann} lacks the assumption of characteristic zero, which is clearly necessary for the exponential function to be well-defined. For a discussion of obstructions in nonzero characteristic, which is a little more involved, we direct the reader to \cite{gerstenhaber2}.
\end{remark}

\begin{remark}
Note that the above proposition does not classify arbitrary equivalences up to order $k + 1$, but only those with this specific form. However, in the case $k = 0$, where we consider \emph{infinitesimal} deformations, every equivalence up to order $1$ is of the form $\exp(\lambda [T_1,\cdot])$, and the only additional requirement is that $T_1$ be invariant. 

This proposition also lays out a deeper understanding for the difference of the two notions of invariance of Hochschild cohomology: In the case
$\HH^2(\As)^G = 0$, all invariant deformations are equivalent, but not necessarily using invariant equivalences. If, however, $\HH^2_G(\As) = 0$, all invariant infinitesimal deformations are equivalent with invariant equivalences. 
\end{remark}

\section{Results for proper actions}
\label{SectionProperHKR}
\subsection{Injectivity}
Let us now restrict to the case $\N = \As = C^\infty(M)$ of smooth functions on a smooth manifold $M$, where the bimodule structure is given by ordinary multiplication of functions. Let $G$ be a Lie group acting smoothly on $M$. By pullback, this induces an action of $G$ on $C^\infty(M)$, which in turn induces an action on $\HC^\bullet(C^\infty(M))$ as discussed in Section \ref{SectionActionsAndInvariant}. Recall that an action $\acts : G \times M \to M$ on a manifold $M$ is called \emph{proper} if the map 
\begin{align}
\bar \acts : G \times M \to M \times M, \, (g,m) \mapsto (g \acts m,m),
\end{align}
is \emph{proper}, meaning preimages of compact sets are compact under $\bar \acts$. This includes a large class of actions, e.g. actions of compact groups or also the natural action of a Lie group $G$ on a homogeneous space $G/H$ whenever $H$ is a compact Lie subgroup. On manifolds, properness of an action is equivalent to the following statement: for any two convergent sequences $\{x_i\}_{i \in \mathbb{N}}$ and $\{g_i \acts x_i \}_{i \in \mathbb{N}}$, the sequence  $\{g_n\}_{n \in \mathbb{N}}$ has a convergent subsequence (see for example \cite[p.59]{ortega}).

 We will explicitly construct averaging operators for actions of this kind on the space of Hochschild cochains. To construct this operator we will use partitions of unity (see for example \cite[p.13]{ortega}. Recall that for every open cover of a smooth manifold, there exists a smooth partition of unity \emph{with compact support} subordinate to it. 
We further require a $G$-invariant analogue:

\begin{proposition}[Existence of $G$-invariant partitions of unity] \cite[p.61]{ortega}
\label{PropositionInvariantPartition}
Let $\acts : G \times M \to M$ be a proper, smooth action of a Lie group $G$ on a manifold $M$, and let $\{O_\alpha\}_{\alpha \in I}$ be an open cover of $M$ by $G$-invariant subsets. Then there exists a subordinate partition of unity $\{\chi_n\}_{n \in \mathbb{N}}$ consisting of $G$-invariant functions $\chi_n \in C^\infty(M)^G$.
\end{proposition}

Note that from this proposition one does not necessarily gain a partition of unity with functions in $C^\infty_0(M)$, as compact support is in general not compatible with $G$-invariance. Thus the compactly supported partitions and the invariant partitions should be viewed as separate results, and indeed, we will need both in the following.
First, we prove there is a way to average these kinds of cochains:

\begin{lemma}
\label{LemmaAveraging}
Let $G$ be a Lie group acting smoothly and properly on the smooth manifold $M$. Given a left invariant volume form $\Omega \in \Gamma^\infty(\Lambda^{\dim G} T^* G)$, a compactly supported function $\xi \in C^\infty_0(M)$ and a continuous cochain $\beta \in \HC^{k}_\cont(C^\infty(M))$ for $k \in \mathbb{N}_0$, the formula
\begin{equation}
\begin{aligned}
&(\xi \cdot \beta)^\av : C^\infty(M)^{\otimes k} \times M \to \mathbb{R}, \\ 
&(\xi \cdot \beta)^\av(f_1,\dots,f_{k})(p) := 
\int_G (g \acts (\xi \cdot \beta)) (f_1, \dots, f_{k})(p) \Omega(g)  
\end{aligned}
\end{equation}
defines an element $(\xi \cdot \beta)^\av \in \HC^k_{G,\cont}(C^\infty(M))$. This averaging is linear and commutes with the differential in the following sense:
\begin{align}
(\xi \cdot \delta \beta)^\av = \delta(( \xi \cdot \beta)^\av) \quad \forall \beta \in \HC^\bullet_\cont(C^\infty(M)).
\end{align}
\end{lemma}

\begin{proof}
To show that this map is well-defined it suffices to show that the integrand is zero outside of a compact domain. Restricting to an arbitrary open subset $U \subset M$ with compact closure $\overline{U}$, we want to show that the set $G_{U,\xi} \subset G$ of group elements which can have non-vanishing contribution to the integration is compact. Analyze this set:
\begin{align*}
G_{U,\xi} &= \left \lbrace g \in G \mid \exists p \in \overline{U} :  g^{-1} \acts p \in \supp \, \xi \right \rbrace 
\\ &=
\left \lbrace g \in G \mid \exists p \in \overline{U} :  \overline{\acts}(g^{-1},p) \in \supp \, \xi \times \overline{U} \right \rbrace 
\\ &=
\pr_G \left( \left( \overline{\acts} \circ (\inv \times \id) \right)^{-1} (\supp \, \xi \times \overline{U}) \right).
\end{align*}
Note that $\xi$ has compact support, $\inv : G \to G, g \mapsto g^{-1}$ is a homeomorphism, and $\overline{\acts}$ is a proper map, so the argument of the projection $\pr_G : G \times M \to G$ in the above equation is a compact set and the image $G_{U,\xi}$ of the projection is also a compact set. This implies for every $p \in U$
\begin{align}
(\xi \cdot \beta)^\av(f_1,\dots,f_{k-1})(p) = \int_{G_{U,\xi}} (g \acts (\xi \cdot \beta)) (f_1,\dots,f_{k-1})(p) \Omega(g),
\end{align}
wherefore the integral is well-defined as an integral of a smooth function over a compact set. 
Note also that for continuous cochains $\beta$, this averaging still yields con\-tin\-u\-ous co\-chains, which is treated in Lemma \ref{LemmaContinuousAverage}. Furthermore, let us show that the map $(\xi \cdot \beta)^\av(f_1,\dots,f_k)$ is a smooth function for all $f_i \in C^\infty(M)$: Since the unaveraged $(\xi \cdot \beta)(f_1,\dots,f_k)$ is a smooth function, this is implied if the order of integration in the parameter $g \in G$ and differentiation in the parameter $p \in M$ can be reversed. Since around every point we can restrict the integration to a compact domain, the function is continuous in $g$ and smooth in $p$. Then the smoothness follows from the Leibniz integral rule for general measure theoretic spaces.
The facts that this averaging is linear and commutes with the Hochschild differential are straightforward calculations. This concludes the proposition.
\end{proof}

We now proceed to proving the injectivity:

\begin{proposition}
\label{PropInjectiveIota}
For a Lie group $G$ acting properly and smoothly on a manifold $M$, taking the induced actions on $\HC^\bullet(C^\infty(M))$ and $\HH^\bullet(C^\infty(M))$, one finds that the map
\begin{align}
\iota  : \HH^\bullet_{G,\cont}(C^\infty(M)) \to \HH_\cont^\bullet(C^\infty(M))^G
\end{align}
is injective.
\end{proposition}

\begin{remark}
We will see that proving this statement comes down to finding a way to globally average elements of $\HC^\bullet_\cont(C^\infty(M))$ in a way that respects the Hochschild differential. If $G$ was compact, this would be quite standard, since we could choose, with the notation from the Lemma \ref{LemmaAveraging}:
\begin{align*}
\phi \mapsto \phi^\av := \int_G (g \acts \phi) \, \Omega(g).
\end{align*}
For non-compact $G$, this integral is in general not well-defined. But as we will see, proper group actions are built specifically in a way that allows for a more sophisticated averaging procedure.
\end{remark}

\begin{proof}[Proof of Proposition \ref{PropInjectiveIota}] 
Injectivity is equivalent to the following: Given an invariant coboundary $\phi \in \HC^{k+1}_{G,\cont}(C^\infty(M))$ with
\begin{align}
\phi = \delta \psi, \hspace{2mm} \psi \in \HC_\cont^{k}(C^\infty(M)),
\end{align}
there must exist an invariant $\tilde \psi \in \HC^{k}_{G,\cont}(C^\infty(M))$ with $\phi = \delta \tilde \psi$. 

First, choose a countable open cover $\{O_n\}_{n \in \mathbb{N}}$ of $M$ with compact closure  $\overline{O_n}$ for all $n \in \mathbb{N}$ and a subordinate partition of unity $\{\xi_n\}_{n \in \mathbb{N}}$ with
\begin{align}
\supp \, \xi_n \subset O_n.
\end{align}
Choose any left invariant volume form so that we can use Lemma \ref{LemmaAveraging} to define the averages of $\xi_n \cdot \psi \in \HC^k(C^\infty(M))$ and $\xi_n \cdot 1 = \xi_n \in C^\infty(M) = \HC^0(C^\infty(M))$, so
\begin{align}
(\xi_n \cdot \psi)^\av, \quad  \xi_n^\av := (\xi_n \cdot 1)^\av.
\end{align}

Furthermore, define the sets $U_n := \left(\xi_n^\av\right)^{-1} ( (0,\infty) )$. They are open and $G$-invariant. Also, since for every $p \in M$ there exists a $\xi_n$ with $\xi_n(p) \neq 0$ and thus $\xi_n^\av(p) \neq 0$, so the $U_n$ cover $M$. Thus by Proposition \ref{PropositionInvariantPartition} there exists a $G$-invariant partition of unity $\{\chi_n\}_{n \in \mathbb{N}}$ subordinate to the $U_n$, satisfying $\supp \, \chi_n \subset U_n$. In particular this means that for every $p \in M$ there exists a $\chi_n$ with $\chi_n(p) \neq 0$ and by construction also $\xi^\av_n(p) \neq 0$.

Using these functions, $\phi = \delta \psi$ implies
\begin{align}
\xi_n \cdot \phi &= \xi_n \cdot \delta \psi = \delta( \xi_n \cdot \psi).
\end{align}
This equation can now be averaged; as $\phi$ is already invariant, and $\delta$ commutes with the averaging integral, it follows that
\begin{align}
 \xi_n^\av \cdot \phi = (\xi_n \cdot \phi)^\av  = (\delta (\xi_n \cdot \psi))^\av = \delta ((\xi_n \cdot \psi)^\av) .
\end{align}
This cannot yet be summed over $n$, as neither side has to be locally finite. However, multiplication with elements $\chi_n$ of the $G$-invariant partition of unity yields a well-defined, locally finite sum:
\begin{align}
\sum_{n \in \mathbb{N}} \chi_n \cdot \xi_n^\av \cdot \phi = \sum_{n \in \mathbb{N}} \delta( \chi_n \cdot (\xi_n \cdot \psi)^\av).
\end{align}
As such, the sum can be interchanged with the linear operator $\delta$, giving the equation
\begin{align}
\phi  \cdot \sum_{n \in \mathbb{N}} \chi_n \cdot \xi_n^\av = \delta \left( \sum_{n \in \mathbb{N}} \chi_n \cdot (\xi_n \cdot \psi)^\av \right).
\end{align}
Now $\sum_{n \in \mathbb{N}} \chi_n \cdot \xi_n^\av > 0$, as for every $p \in M$ there exists an  $n \in \mathbb{N}$ with $\chi_n(p) > 0 $ and $\xi_n(p) > 0$, so $\xi_n^\av(p) > 0$. It follows that
\begin{align}
\phi = \delta \left( \frac{\sum_{n \in \mathbb{N}} \chi_n \cdot (\xi_n \cdot \psi)^\av}{\sum_{n \in \mathbb{N}} \chi_n \cdot \xi_n^\av} \right).
\end{align}
By Lemma \ref{LemmaContinuousLocallyFiniteSum} the cochain $\tilde \psi := \frac{\sum_{n \in \mathbb{N}} \chi_n \cdot (\xi_n \cdot \psi)^\av}{\sum_{n \in \mathbb{N}} \chi_n \cdot \xi_n^\av}$ is continuous. As $\tilde \psi$ now only consists of $G$-invariant functions and cochains, the statement is shown.
\end{proof}

\begin{remark}
The above statement also holds when $\cont$ is replaced by $\loc,\diff$ or $\ncdiff$; one only needs to check that the averaging procedure $\psi \mapsto \psi^\av$ leaves these properties untouched, they then carry through the rest of the construction without problem.
\end{remark}

%
%

\subsection{Surjectivity} For the surjectivity of the natural map $\iota$, one would have to show that every invariant class contains an invariant cochain. In analogy to the proof of injectivity, the intuitive route would be to take an arbitrary cochain from the invariant class and average it to gain an invariant one. However, it is not clear why the averaged cochain is still in the same class. With the notation of the proof for injectivity, one would have to show that
\begin{align}
\int_G (g \acts \xi_n) (\phi - g \acts \phi) \Omega(g)
\end{align}
is a coboundary if the expression $\phi - g \acts \phi$ is a coboundary for all $g \in G$. While it is true that for every $g \in G$ there exists a cochain $\psi_g$ so that $\phi - g \acts \phi = \delta \psi_g$, there is no control over the map $g \mapsto \psi_g$.

We will take another route using the celebrated Hochschild-Kostant-Rosenberg theorem, calculating Hochschild cohomology of the algebra of smooth functions in the non-invariant case; see \cite{HKR} for its original formulation in terms of algebraic varieties. The following version is phrased in terms of smooth manifolds.

We will use some standard notation from differential geometry: If $\pi: E \to M$ is a smooth vector bundle, denote by $\Gamma^\infty(E)$ its smooth sections, and by $\Lambda^k E$ its $k$-th exterior power. For a smooth section $X \in \Gamma^\infty(E)$ and any $p \in M$, we denote by $X_p$ its value at the point $p$. The map $d$ is the usual Cartan differential and $\ih_\alpha$ for some smooth form $\alpha \in \Gamma^\infty(T^*M)$ denotes the interior product map, i.e. contraction with the form $\alpha$. For a smooth map $f \in C^\infty(M)$ and a vector field $X \in \Gamma^\infty(TM)$, the expression $X(f) \in C^\infty(M)$ denotes $X$ acting on $f$ as a derivation in the canonical way. We also denote the space of multivector fields by $\mathfrak{X}^\bullet(M) =
 \bigoplus_{k=0}^\infty \Gamma^\infty(\Lambda^k TM)$.

Let us now recall the Hochschild-Kostant-Rosenberg theorem, whose precise formulation we cite from \cite{waldmann}. Proofs are found in \cite{kontsevich2} for the differential case, \cite{cahen} for the local case, and \cite{nadaud1999continuous} and \cite{pflaum1998continuous} for the continuous case. Before the full result was known for continuous cochains, Connes proved the special case of compact $M$ in \cite{connes1985noncommutative} and Gutt described degree 2 cohomology in \cite{gutt1997some}.

\begin{theorem}[Hochschild-Kostant-Rosenberg (HKR)] \cite[p.417]{waldmann}
\label{TheoremHKR}
Let $M$ be a smooth manifold. Define
\begin{equation}
\label{FormulaHKR}
\begin{aligned}
U :  \mathfrak{X}^\bullet(M)  &\to  \HC^\bullet_{\cont} (C^\infty(M)), \hspace{3mm} X \mapsto U(X),\\
U(X)&(f_1,\dots,f_k) = \frac{1}{k!} \ih_{d f_k} \dots \ih_{d f_1} X.
\end{aligned}
\end{equation}
This induces an isomorphism
\begin{align}
\mathcal{U} : \mathfrak{X}^\bullet(M) \to \mathrm{HH}_{\cont}^\bullet(C^\infty(M)),
\end{align}
where $\cont$ can be replaced with $\loc$, $\diff$, $\ncdiff$.
\end{theorem}

\begin{remark}
For factorizing multivector fields $X_1 \wedge \dots \wedge X_n$, the above definition yields
\begin{align}
U(X_1 \wedge \dots \wedge X_n)(f_1,\dots,f_n) = \frac{1}{k!} \sum_{\sigma \in S_k} X_{\sigma(1)} (f_1) \dots X_{\sigma(k)} (f_k),
\end{align}
where $X_i(f_j)$ is just the derivation associated to $X_i$ acting on $f_j$ and $S_k$ denotes the permutation group on $k$ letters. Note that the Serre-Swan theorem of differential geometry implies that this formula fully defines the map on all multivector fields.
\end{remark}

Note that while the proof of injectivity is fairly straightforward, the surjectivity of this map heavily relies on the analytical/topological structure of the Hochschild cochains and is highly technical. To the knowledge of the author, neither a proof nor a counterexample to the HKR theorem for non-continuous Hochschild cohomology has been found.

We can define a group action on the space of multivector fields which corresponds to the action on the Hochschild complex. Define for every $g \in G$ and $X \in \Gamma^\infty(TM)$ the action
\begin{align}
(g \acts X)(f):= g \acts (X(g^{-1} \acts f))
\end{align}
and accordingly actions on higher order multivector fields via
\begin{align}
g \acts (X \wedge Y) := (g \acts X) \wedge (g \acts Y).
\end{align} 
Using the explicit formula of $U$ above, we find
\begin{align}
g \acts U(X) = U(g \acts X)
\end{align}
for all $g \in G$.

\begin{proposition}
\label{PropSurjectiveIota}
For a smooth action of $G$ on $M$, using the induced actions on $\HC^\bullet(C^\infty(M))$ and $\HH^\bullet(C^\infty(M))$, the natural map
\begin{align}
\iota: \HH_{G,\cont}^\bullet(C^\infty(M)) \to \HH_\cont^\bullet(C^\infty(M))^G
\end{align}
is surjective. Here, $\cont$ can be replaced with $\loc,\diff,\ncdiff$.
\end{proposition}

\begin{proof}
Recall that surjectivity of $\iota$ means that every invariant class contains an invariant cocycle. Given any $[\phi] \in \HH^\bullet(C^\infty(M))^G$, for all $g \in G$ there exists some $\psi_g \in \HC^\bullet(C^\infty(M))$ so that
\begin{align}
g \acts \phi = \phi + \delta(\psi_g).
\end{align} 
By the HKR theorem, there exists a unique multivector field $X \in \mathfrak{X}^\bullet(M)$ and some $\xi \in \HC^\bullet(C^\infty(M))$ with
\begin{align}
U(X) = \phi + \delta \xi.
\end{align}
Now, for any $g \in G$ we find
\begin{equation}
\begin{aligned}
U(g \acts X) 
&= 
g \acts U(X)
= 
g \acts \phi + g \acts \delta \xi 
\\ &=
\phi + \delta (\psi_g + g \acts \xi)
=
U(X) + \delta(\psi_g + g \acts \xi - \xi).
\end{aligned}
\end{equation}
This means that $U(X)$ and $U(g \acts X)$ lie in the same equivalence class, which, by injectivity of the HKR isomorphism $\mathcal{U}$, is only possible if $X = g \acts X$. It follows that $U(X)$ is an invariant cocycle lying in the same equivalence class as $\phi$, which proves surjectivity of $\iota$.
\end{proof} 

To summarize: If we assume a proper group action and restrict to any one of the analytical subcomplexes of $\HH^\bullet(C^\infty(M))$ where the HKR theorem is applicable, we gain isomorphy of the spaces:

\begin{theorem}
\label{TheoremProperIsomorphism}
For a proper, smooth action of a Lie group $G$ on a manifold $M$, taking the induced actions on $\HC^\bullet(C^\infty(M))$ and $\HH^\bullet(C^\infty(M))^G$, one finds that the natural map
\begin{align}
\iota: \HH_{G,\cont}^\bullet(C^\infty(M)) \to \HH_\cont^\bullet(C^\infty(M))^G, \, [\phi] \mapsto [\phi]
\end{align}
is a well-defined isomorphism. Here, $\cont$ can be replaced by $\loc,\diff,\ncdiff$.
\end{theorem}

\subsection{Invariant multivector fields and an invariant HKR map}

We briefly want to look into an invariant analogue of the HKR map. Let us first define the space of invariant multivector field, motivated by the action on multivector fields which we derived earlier:

\begin{definition}[$G$-invariant multivector fields]
For a smooth group action $\Phi^M$ of a Lie group $G$, define the \emph{$G$-invariant multivector fields on $M$} by
\begin{align}
\mathfrak{X}^\bullet(M)^G :&= \left \lbrace X \in \mathfrak{X}^\bullet(M) \mid g \acts X = X \quad \forall g \in G \right \rbrace.
\end{align}
\end{definition}

With our previous results, we are now able to state an invariant version of the HKR theorem which relates invariant multivector fields to both of our notions of invariant Hochschild cohomology.

\begin{theorem}[Invariant HKR]
Given a smooth action of a group $G$ on a manifold $M$, using the induced actions on $\mathfrak{X}^\bullet(M)$ and $\HH^\bullet(C^\infty(M))$, the HKR map $U$ from Equation (\ref{FormulaHKR}) induces an isomorphism (of vector spaces/super Lie algebras/Gerstenhaber algebras) to the space of invariant classes
\begin{align}
\mathcal{U}^G : \mathfrak{X}^\bullet(M)^G \to \HH^\bullet_{\cont} (C^\infty(M))^G.
\end{align}
If $G$ acts properly on $M$, it further induces an isomorphism (of vector spaces/super Lie algebras/Gerstenhaber algebras)
\begin{align}
\mathfrak{X}^\bullet(M)^G \to \HH^\bullet_{G,\cont} (C^\infty(M)).
\end{align}
Here, $\cont$ can be replaced with $\loc,\diff,\ncdiff$.
\end{theorem}

\begin{proof}
As shown previously, the HKR isomorphism $\mathcal{U}$ is equivariant with respect to the actions of $G$ on $\HH^\bullet(\As)$ and $\mathfrak{X}^\bullet(M)$, hence it restricts to an isomorphism (of vector spaces, and by compatibility of the group action with the unmentioned, necessary operations an isomorphism of super Lie algebras and Gerstenhaber algebras as well) $\mathfrak{X}^\bullet(M)^G \to \HH^\bullet_{\cont} (C^\infty(M))^G$. \\
If the action is proper, by Theorem \ref{TheoremProperIsomorphism} the map 
\begin{align}
\iota : \HH^\bullet_{G, \cont}(C^\infty(M)) \to \HH^\bullet_\cont(C^\infty(M))^G
\end{align}
is an isomorphism, and thus $\iota^{-1} \circ \mathcal{U}^G$ gives us the desired isomorphism $\mathfrak{X}^\bullet(M)^G \to \HH^\bullet_{G,\cont} (C^\infty(M))$. This concludes the proof.
\end{proof}

\begin{example}[Homogeneous spaces]
By the invariant HKR Theorem, we see that whenever the given action is smooth, proper and \emph{transitive}, the invariant cohomology $\HH^\bullet_{G,\cont}(C^\infty(M))$ can be identified with a subspace of $\Lambda^\bullet T_p M$ for an arbitrary point $p \in M$, as an invariant vector field is in this case already fully determined by its value at a single point. As such, the cohomology spaces become finite-dimensional.

In particular, when $M$ is represented as a homogeneous space, i.e. $M = G/H$ with $H$ a closed Lie subgroup of the Lie group $G$, one has a proper action of $G$ on $M$ if and only if $H$ is compact: stabilizers of proper actions must be compact, and proving the other direction is a straightforward exercise using the definition of properness via sequences in $G$ and $G/H$. 

In this case, if $\mathfrak{g}$ and $\mathfrak{h}$ are the corresponding Lie algebras of $G$ and $H$, we have
\begin{align}
\HH_{G,\cont}^\bullet(C^\infty(G/H)) \cong \left( \Lambda^\bullet \mathfrak{g}/\mathfrak{h} \right)^H,
\end{align}
where the $H$ invariance is to be understood with respect to the adjoint action of $G$ restricted to $H$.
\end{example}
\appendix
\section{Continuity of properly averaged cochains}

\label{SectionLocallyConvex}
In this appendix, we want to show that the averaging procedures we defined in Section \ref{SectionProperHKR} map continuous cochains to continuous cochains. The notion of continuity is here induced by the topology of uniform convergence on compact subsets on $C^\infty(M)$.
The necessary theory about this locally convex topology can, for example, be found in \cite{jarchow,meise2013einfuhrung,rudin1991functional}. We shortly recall the construction of the seminorms on this space.

For any compact set $K$ lying within a chart $(U,x)$ and any $l \in \mathbb{N}$ define the corresponding seminorms as
\begin{align}
p_{k,K,U,x}(f) := \max_{p \in K, i_1, \dots, i_k} \left| \frac{\partial^k (f \circ x^{-1})}{\partial x^{i_1} \dotsb \partial x^{i_k}} (x(p)) \right|.
\end{align}
For details see for example \cite{hirsch2012differential}. Note that if our manifold is equipped with a Lie group action, we can also define an action on the space of seminorms: for all $f \in C^\infty(M)$ set
\begin{align}
\label{EquationChangingTrickSeminorms}
(g \acts p_{k,K,U,x})(g \acts f) := p_{k,K,U,x}(f).
\end{align} This implies:
\begin{align}
g \acts p_{k,K,U,x} = p_{k,g \, \acts \, K, g \, \acts \, U, g \, \acts \, x}
\end{align}
where the diffeomorphism $M \to M, p \mapsto g  \acts p$ carries the chart $(U,x)$ to another chart $(g \acts U, g \acts x)$.
This notational trick will greatly simplify some of the expressions in the following proof:

\begin{lemma}[Continuity of averaged cochains]
\label{LemmaContinuousAverage}
Let $R$ be a compact subset of a Lie group $G$ with an action $\Phi^M : G \times M \to M$ on a smooth manifold $M$. For a continuous $\psi \in \HC_\cont^n(C^\infty(M))$, the map $\psi^\av := \int_R (g \acts \psi) \Omega(g)$ is also continuous, where the integration is performed with respect to some invariant measure as in Section \ref{SectionProperHKR}.
\end{lemma}
\begin{proof}
Continuity of $\psi$ is equivalent to the following: for every seminorm $p_{k,K,U,x}$ there exist finitely many seminorms $p_{l_i,L_i,U_i,x_i}$ with $i = 1, \dots, n$ and a constant $C \geq 0$ so that
\begin{align}
p_{k,K,U,x}(\psi(f_1,\dots,f_n))  \leq C \cdot  p_{l_1,L_1,U_1,x_1} (f_1) \dotsm p_{l_n,L_n,U_n,x_n}(f_n)
\end{align}

Then, for every $g \in R$, there also exist such $l_i, L_i,U_i,x_i,C$ so that

\begin{align}
\begin{split}
p_{k,K,U,x}((g \acts \psi) (f_1,\dots,f_n)) 
 &= 
(g^{-1} \acts p_{k,K,U,x}) (\psi (g^{-1} \acts f_1,\dots,g^{-1} \acts f_n))
\\  &\leq C \cdot 
p_{l_1, L_1, U_1, x_1}(g^{-1} \acts f_1) \dotsm
p_{l_n, L_n, U_n, x_n}(g^{-1} \acts f_n) 
\\ &=
(g \acts p_{l_1,L_1,U_1,x_1}) (f_1) \dotsm
(g \acts p_{l_n,L_n,U_n,x_n}) (f_n). \label{InequalityContinuityAverage}
\end{split}
\end{align}
Hence, $g \acts \psi$ is a continuous map.
Similar inequalities can be used for the averaging procedure:

\begin{align*}
p_{k,K,U,x}(\psi^\av (f_1,\dots,f_n)) &=
\max_{\substack{p \in K \\ i_1, \dots, i_k \in \mathbb{N} }} \left| \int_R \frac{\partial^k (g \acts \psi)(f_1,\dots, f_n)\circ x^{-1}}{\partial x^{i_1} \dots \partial x^{i_k}} (x(p)) \Omega(g) \right|\\
& \leq
\int_R
\max_{\substack{p \in K \\ i_1, \dots, i_k \in \mathbb{N} }} \left| \frac{\partial^k (g \acts \psi)(f_1,\dots, f_n)\circ x^{-1}}{\partial x^{i_1} \dots \partial x^{i_k}} (x(p))  \right| \Omega(g)
\\ &=
\int_R p_{k,K,U,x}((g \acts \psi)(f_1, \dots, f_n)) \Omega(g)
\\ &\leq
\Vol(R) \cdot \max_{g \in R} p_{k, K,U,x}((g \acts \psi)(f_1, \dots,f_n)) 
\\ &\leq
C \cdot \Vol(R) \cdot
\max_{g \in R} 
(g \acts p_{l_1, L_1, U_1, x_1}) (f_1) \dotsm\\ & \phantom{ohjemine heribert und maria} \dotsm
(g \acts p_{l_n, L_n,U_n, x_n}) (f_n),
\end{align*}
where the last inequality is due to Inequality (\ref{InequalityContinuityAverage}). Note that the last maximum is indeed a maximum, because using Equation (\ref{EquationChangingTrickSeminorms}), we can move the $g$-dependence of the seminorm into its argument, so that the expression becomes a continuous function in $g$. Now, for every compact set $g \acts L_i$ in a coordinate patch, using an exhaustion of the coordinate patch by compact sets, one can construct a set $\widehat L_{i}(g)$, so that $g \acts L_i$ lies in the interior of $\widehat L_i(g)$ and $\widehat L_i(g)$ is still a subset of the coordinate patch. The interiors of the $\widehat L_i(g)$ for all $g \in R$ then yield an open cover of $R \acts L_i$, so the maximum over all the $g \in R$ in the above chain of inequalities is assumed in one of the $\widehat L_i(g)$. However, as $R \acts L_i$ is a compact set, finitely many $\widehat L_i(g_{t_i}), t_i = 1, \dots, r_i$ are sufficient to cover $R \acts L_i$, so the maximum is assumed in one of the finitely many $\widehat L_i(g_{t_i})$. The choice of the $\widehat L_i(g_{t_i})$ does not depend on the $f_1, \dots, f_n$, so one finds for every $i = 1, \dots, n$
\begin{align}
\max_{g \in R} (g \acts p_{l_i, L_i,U_i,x_i}) (f_i) \leq \max_{t_i = 1, \dots, r_i} p_{l_i, \widehat L_i(g_{t_i}), g_{t_i} \acts U_i, g_{t_i} \acts x_i} (f_i)
\end{align}
and finally
\begin{align*}
p_{k,K,U,x}(\psi^\av (f_1,\dots,f_n)) \leq 
C \cdot \Vol(R) \cdot 
 &\left(\max_{t_1 = 1, \dots, r_1}p_{l_1, \widehat L_1(g_{t_1}), g_{t_1} \acts U_1, g_{t_1} \acts x_1} (f_1) \right) \dotsm \\  &
 \dotsm \left( \max_{t_n = 1, \dots, r_n}p_{l_1, \widehat L_n(g_{t_n}), g_{t_n} \acts U_n, g_{t_n} \acts x_n} (f_n) \right) .
\end{align*}
Finite maxima of seminorms are again seminorms, and multiplication of a seminorm with non-negative constants again yields a seminorm. As such, continuity of the averaged map is shown.
\end{proof}

\begin{lemma}[Continuity of locally finite sums]
\label{LemmaContinuousLocallyFiniteSum}
Consider a partition of unity $\{\chi_i\}_{i \in I}$ of a smooth manifold $M$. Then the locally finite sum
\begin{align}
\sum_{i \in I} \chi_i \cdot \psi_i
\end{align}
of continuous cochains $\psi_i \in \HC^n(C^\infty(M))$ is again a continuous cochain. 
\end{lemma}
\begin{proof}
The calculation will in the following only be done for seminorms in zeroth order of differentiation. In higher orders of differentiation one only receives factors with the corresponding maxima of the derivatives of $\chi_i$, and via product rule one receives multiple summands, for which just the same calculations can be done. 
We also drop all references to charts $(U,x)$, as they do not play a role in the following. Let $K$ be some compact set within a coordinate patch
\begin{align}
p_{0,K} \left( \sum_{i \in I} \chi_i \cdot \psi_i(f_1,\dots,f_n) \right) =
\max_{x \in K} \left| \sum_{i} \chi_i(x) \psi_i(f_1,\dots,f_n)(x) \right|.
\end{align}
For every $x \in K$, because of the locally finiteness, there exists an open neighbourhood $U_x \subset K$, so that the sum is finite on $U_x$. The union of these neighbourhoods is an open cover of $K$, by compactness of $K$ there then exist finitely many $U_{x_j}$, $j = 1, \dots, m$, so that $K = \bigcup_{j=1}^n U_{x_j}$. It follows that the maximum is assumed in one of the finitely many $U_{x_j}$, and by consequence in one of the finitely many $\overline{U_{x_j}}$, which are compact as a closed subset of the compact $K$. The largest index $l$, so that the $\chi_l$ do not vanish on $U_{x_j}$ is denoted by $l_j$, with which one can write:
\begin{align*}
p_{0,K} \left( \sum_{i \in I} \chi_i \cdot \psi_i(f_1,\dots,f_n) \right) 
& =
\max_{j =1, \dots, n} \max_{x \in \overline{U_{x_j}}} \left| \sum_{i=1}^{l_j} \chi_i(x) \cdot \psi_i(f_1,\dots,f_n)(x) \right|
 \\ 
 & \leq
\max_{j =1, \dots, n} \max_{x \in \overline{U_{x_j}}} \max_{i=1,\dots,l_j}  \left|\psi_i(f_1,\dots,f_n)(x) \right| 
\\
& =
\max_{j =1, \dots, n} \max_{i=1,\dots,l_j}  p_{0, \overline{U_{x_j}}} (\psi_i(f_1,\dots,f_n) ) .
\end{align*}
Note that the choice of $U_{x_j}$ does not depend on the $f_i$. As finite maxima of seminorms yield seminorms again, the above inequality proves continuity of the locally finite sum.
\end{proof}

\section*{Acknowledgments}
The author would like to thank Dr. Bas Janssens from the TU Delft and Prof. Dr. Stefan Waldmann, for their continued support and instructive explanations during the writing of this document. Thanks also to Dr. Matthias Sch{\"o}tz and the anonymous referees for their constructive criticism. The author was supported by the NWO Vidi Grant 639.032.734 ``Cohomology and representation theory of infinite-dimensional Lie groups''.      



\begin{thebibliography}{99}

\bibitem[BDN17]{brodzki2017periodic}
J. Brodzki, S. Dave, and V. Nistor:
\emph{The periodic cyclic homology of crossed products of finite type algebras}, 
Adv. Math. (N Y) {\bf 306} (2017) 494-523.

\bibitem[BFF+77]{bayen1977quantum}
F. Bayen, M. Flato, C. Fronsdal, A. Lichnerowicz, and D. Sternheimer: \emph{Quantum mechanics as a deformation of classical mechanics}, Lett. Math. Phys. {\bf 6} (1977) 521--530.

\bibitem[BG94]{block1994equivariant}
J. Block, E. Getzler: \emph{Equivariant cyclic homology and equivariant differential forms},
Ann. Sci. {\'E}c. Norm. Sup{\'e}r. {\bf 27} (1994) 493--527.

\bibitem[Bry87a]{brylinski1987algebras}
J.-L. Brylinski:
\emph{Algebras associated with group actions and their homology}, 
Brown preprint (1987).

\bibitem[Bry87b]{brylinski1987cyclic}
J.-L. Brylinski: \emph{Cyclic homology and equivariant theories}, 
Ann. Inst. Fourier {\bf 37} (1987) 15--28.

\bibitem[CGDW80]{cahen}
M. Cahen, S. Gutt, and M. De Wilde: 
\emph{Local cohomology of the algebra of {C}$^\infty$ functions on a connected manifold},
Lett. Math. Phys. {\bf 4} (1980) 157--167.

\bibitem[Con85]{connes1985noncommutative}
A. Connes: \emph{Non-commutative differential geometry},
Publ. Math. IH{\'E}S {\bf 62} (1985) 41--144.

\bibitem[Fed94]{fedosov1994simple}
B. V. Fedosov:
\emph{A simple geometrical construction of deformation quantization},
J. Differ. Geom. {\bf 40} (1994) 213--238.

\bibitem[Ger63]{gerstenhaber1}
M. Gerstenhaber:
\emph{The cohomology structure of an associative ring},
Ann. Math. {\bf 78} (1963) 267--288.

\bibitem[Ger64]{gerstenhaber2}
M. Gerstenhaber: \emph{On the deformation of rings and algebras},
Ann. Math. {\bf 79} (1964) 59--103.

\bibitem[Ger66]{gerstenhaber3}
M. Gerstenhaber: \emph{On the deformation of rings and algebras: {II}},
Ann. Math. {\bf 84} (1966) 1--19.

\bibitem[Ger68]{gerstenhaber4}
M. Gerstenhaber:
\emph{On the deformation of rings and algebras: {III}},
Ann. Math. {\bf 88} (1968) 1--34.

\bibitem[Ger74]{gerstenhaber5}
M. Gerstenhaber:
\emph{On the deformation of rings and algebras {IV}},
Ann. Math. {\bf 99} (1974) 257--276.

\bibitem[Gut97]{gutt1997some}
S. Gutt: \emph{On some second Hochschild cohomology spaces for algebras of functions on a manifold},
Lett. Math. Phys. {\bf 39} (1997) 157--162.

\bibitem[Hir94]{hirsch2012differential}
M. W. Hirsch:
\emph{Differential Topology}, 
Graduate Texts in Mathematics {\bf 33}, Springer-Verlag, New York (1994).

\bibitem[HKR62]{HKR}
G. Hochschild, B. Kostant, and A. Rosenberg:
\emph{Differential forms on regular affine algebras},
Trans. Am. Math. Soc. {\bf 102} (1962) 383--408.

\bibitem[Hoc45]{hochschild1945cohomology}
G. Hochschild:
\emph{On the cohomology groups of an associative algebra},
Ann. Math. {\bf 46} (1945) 58--67.

\bibitem[Jar81]{jarchow}
H. Jarchow: {Locally convex spaces}, 
Mathematische Leitf\"aden, Teubner, Stuttgart (1981).

\bibitem[Kon97]{kontsevich1}
M. Kontsevich:
\emph{Formality conjecture}, in:
Deformation theory and symplectic geometry, Proc. Conf. held in Ascona, June 1996,
D. Sternheimer, J. Rawnsley and S. Gutt (eds.), Math. Phys. Stud. 20, Kluwer Academic Publishers, Dordrecht (1997) 139--156.

\bibitem[Kon03]{kontsevich2}
M. Kontsevich:
\emph{Deformation quantization of {P}oisson manifolds},
Lett. Math. Phys. {\bf 66} (2003) 157--216.

\bibitem[MV92]{meise2013einfuhrung}
R. Meise, D. Vogt:
\emph{Einf\"uhrung in die {F}unktionalanalysis}, Aufbaukurs Mathematik  {\bf 62},
Vieweg Studium, Braunschweig (1992).

\bibitem[Nad99]{nadaud1999continuous}
F. Nadaud:
\emph{On continuous and differential hochschild cohomology},
Lett. Math. Phys. {\bf 47} (1999) 85--95.

\bibitem[NPPT06]{neumaier2006homology}
N. Neumaier, M.J. Pflaum, H.B. Posthuma, and X. Tang:
\emph{Homology of formal deformations of proper {\'e}tale lie groupoids},
J. Reine Angew. Math. {\bf 2006} (2006) 117--168.

\bibitem[OR04]{ortega}
J.-P. Ortega, T. S. Ratiu:
\emph{Momentum maps and {H}amiltonian reduction}, Progress in Mathematics {\bf 222},
Birkh\"auser Boston, Inc., Boston (2004).

\bibitem[Pfl98]{pflaum1998continuous}
M. Pflaum:
\emph{On continuous {H}ochschild comology and cohomology groups},
Lett. Math. Phys. {\bf 44} (1998) 43--51.

\bibitem[Rud91]{rudin1991functional}
W. Rudin:
\emph{Functional analysis}, International Series in Pure and Applied Mathematics,
McGraw-Hill, Inc., New York (1991).

\bibitem[Wal07]{waldmann}
S. Waldmann:
\emph{Poisson-Geometrie und Deformationsquantisierung: Eine
  Einf{\"u}hrung}, Springer-Lehrbuch Masterclass, 
Springer, Heidelberg (2007).

\end{thebibliography}
\end{document}